\documentclass{amsart}
\usepackage{MnSymbol}
\usepackage{scalerel}

\usepackage{hyperref}
\hypersetup{
    colorlinks,
    citecolor=black,
    filecolor=black,
    linkcolor=black,
    urlcolor=black
}
\usepackage{comment}

\DeclareFontFamily{OT1}{pzc}{}
\DeclareFontShape{OT1}{pzc}{m}{it}{<-> s * [1.100] pzcmi7t}{}
\DeclareMathAlphabet{\mathpzc}{OT1}{pzc}{m}{it}

    \usepackage{etoolbox}
    \patchcmd{\section}{\scshape}{\large\bfseries}{}{}
    \makeatletter
    \renewcommand{\@secnumfont}{\bfseries}
    \makeatother

\usepackage{tikz-cd}
\tikzcdset{arrow style=math font}
\tikzcdset{scale cd/.style={every label/.append style={scale=#1},
    cells={nodes={scale=#1}}}}

\usepackage[maxbibnames=99]{biblatex}
\usepackage{booktabs}
\bibliography{references}

\numberwithin{equation}{section}
\newtheorem{theorem}{Theorem}[section]
\newtheorem*{theorem*}{Theorem}
\newtheorem{corollary}[theorem]{Corollary}
\newtheorem{lemma}[theorem]{Lemma}
\newtheorem{proposition}[theorem]{Proposition}
\theoremstyle{definition}

\newtheorem*{question*}{Question}

\newtheorem{remark}[theorem]{Remark}

\def\ZZ{\mathbb{Z}}
\def\AA{\mathcal{A}}
\def\CC{\mathcal{C}}
\def\mono{\rightarrowtail}
\def\epi{\twoheadrightarrow}

\def\Ker{\mathrm{Ker}}
\def\Im{\mathrm{Im}}
\def\Coker{\mathrm{Coker}}

\let\oldtocsection=\tocsection 
\let\oldtocsubsection=\tocsubsection 
\renewcommand{\tocsection}[2]{\hspace{0mm}\oldtocsection{#1}{#2}}
\renewcommand{\tocsubsection}[2]{\hspace{1em}\oldtocsubsection{#1}{#2}}

\title[Transfinite version of the Mittag-Leffler condition]{Transfinite version of the Mittag-Leffler condition for the vanishing of the derived limit}

\author{Mishel Carelli} 
\email{mishelcarelli@gmail.com, mishelc@campus.technion.ac.il}
\address{Technion - Israel Institute of Technology}

\author{Sergei O. Ivanov} 
\email{ivanov.s.o.1986@gmail.com, ivanov.s.o.1986@bimsa.cn}
\address{
Beijing Yanqi Lake Institute of Applied Mathematics (BIMSA)}

\thanks{The first named author is supported by Department of Mathematics,
The Technion, Haifa, Israel.
The second named author is supported by Yanqi Lake Beijing Institute of Mathematical Sciences and Applications (BIMSA)}

\begin{document}

\begin{abstract} 
We give a necessary and sufficient condition for an inverse sequence $S_0 \leftarrow S_1 \leftarrow \dots$ indexed by natural numbers to have ${\rm lim}^1S=0$. This condition can be treated as a transfinite version of the Mittag-Leffler condition. We consider inverse sequences in an arbitrary abelian category having a generator and satisfying Grothendieck axioms ${\rm (AB3)}$ and ${\rm (AB4^*)}.$ We also show that the class of inverse sequences $S$ such that ${\rm lim}\: S={\rm lim}^1 S=0$ is the least class of inverse sequences containing the trivial inverse sequence and closed with respect to small limits and a certain type of extensions.
\end{abstract}

\maketitle

\section*{Introduction}

The study of the right derived functors of the functor of limit was initiated in the works of Yeh, Milnor, Roos and Grothendieck \cite{yeh1959higher, milnor1962axiomatic,roos1961foncteurs,grothendieck1961ega} (see also \cite{nobeling1962derivierten,jensen1970vanishing,laudal1965limites}). In particular, it was shown that for inverse sequences of abelian groups, $\lim^i=0$ for $i>1$, but the functor $\lim^1$ turned out to be non-trivial. This functor is referred to as the derived limit. Milnor emphasized the significant role of this functor in algebraic topology by introducing what is now known as Milnor exact sequence for homotopy groups. It was also proven that if the inverse sequence consists of epimorphisms, the derived limit is trivial. Furthermore, Grothendieck in \cite{grothendieck1961ega} introduced the Mittag-Leffler condition for an inverse sequence, which generalizes the epimorphism condition, and proved that it also implies the vanishing of the derived limit.

If all components of an inverse sequence are at most countable abelian groups, then the Mittag-Leffler condition becomes necessary and sufficient for the vanishing of the derived limit. However, for arbitrary inverse sequences of abelian groups, it is not a necessary condition. There arose a need to find a necessary and sufficient variant of this condition. In \cite{emmanouil1996mittag}, Emmanouil showed that for an inverse sequence $S$ consisting of $R$-modules, the Mittag-Leffler condition is equivalent to $\mathrm{lim}^1(S\otimes_R M)=0$ for any $R$-module $M$. However, we took a different approach and began to study how the Mittag-Leffler condition can be modified to become both necessary and sufficient for the derived limit to be zero.

It is meaningful to consider the derived limit in arbitrary abelian categories with small limits. From now on, we will always assume that $\mathcal{A}$ is an abelian category with small direct sums (Grothendieck axiom ${\rm AB3}$), exact small products (Grothendieck axiom ${\rm AB4^*}$), and a generator. Roos proved that in such a category, the derived limit of an epimorphic inverse sequence is trivial \cite{roos2006derived}. Neeman showed that the assumption of having a generator in $\AA$ cannot be omitted \cite{neeman2002counterexample}.

Our work is devoted to description of inverse sequences with trivial derived limit in the category $\AA$. The main result of our work consists of providing a necessary and sufficient condition for the vanishing of the inverse limit, which can be interpreted as a transfinite version of the Mittag-Leffler condition. Among the inverse sequences with a zero derived limit, a special place is occupied by the inverse sequences $L$ for which $\lim L = \lim^1 L = 0$. We call such inverse sequences local. The second main result of this work is that we describe the class of local inverse sequences as the least class containing the trivial inverse sequence and closed under small limits and certain extensions. In this work, we deliberately limit ourselves to considering only inverse sequences indexed by natural numbers, without considering more general posets or categories as indices.

For an inverse sequence  $S$ we denote by $I(S)$ an inverse sequence, consisting of the images $I(S)_i = \Im(S_{i+1} \to S_i).$ Furthermore, we define $I^n(S)$ recursively as $I^{n}(S)=I(I^{n-1}(S)).$ An inverse sequence $S$ satisfies the Mittag-Leffler condition if, for any given $i$, the decreasing sequence of subobjects $S_i \supseteq I^1(S)_i \supseteq I^2(S)_i \supseteq \dots$ stabilizes. This condition is sufficient for the vanishing of the derived limit, but it easy to construct an example that shows that it is not necessarily. For example, if we denote by $\mathbb{Z}_p$ the group of $p$-adic integers, then the inverse sequence $p^i \mathbb{Z}_p$ has a trivial derived limit but it does not satisfy the Mittag-Leffler condition. More generally, if $A=A_0\supseteq A_1 \supseteq ...$ is a complete Hausdorff filtration on an abelian group $A$ (i.e. $A \cong \lim A/A_i$), then $\lim^1 A_i=0$. 

If we consider the completion of an inverse sequence $S$ with respect to the image filtration 
\[ \widehat{S}= \underset{n}{\lim}\: S/I^n(S) \]
then in both the case of the Mittag-Leffler condition and the case of a complete Hausdorff filtration, the morphism of inverse sequences $S \to \widehat{S}$ is an epimorphism. We prove that this morphism is an epimorphism for any inverse sequence with a trivial derived limit. However, it was insufficient for us to consider this completion for formulating a necessary and sufficient condition. So we introduced the concept of $\lambda$-completions of an inverse sequence for any limit ordinal $\lambda$.

For an ordinal $\alpha$ and an inverse sequence $S$, we define $I^\alpha(S)$ such that 
$
I^{\alpha+1}(S)=I(I^\alpha(S))$  and $I^\lambda(S)=\lim_{\alpha<\lambda} I^\alpha(S)$ for a limit ordinal $\lambda.$
It turns out that, despite the inverse sequence being indexed by ordinary natural numbers, this transfinite filtration can stabilize at any ordinal. The $\lambda$-completion of the inverse sequence $S$ is defined as
\[ \widehat S^\lambda = \lim_{\alpha<\lambda} S/I^\alpha(S).\]
We say that an inverse sequence $S$ is \emph{$\lambda$-complete}, if the morphism $S\to \widehat{S}^\lambda$ is an epimorphism. 
The main result of our work is the following theorem (Th. \ref{th:gen_ML}).
\begin{theorem*} Let $\AA$ be an abelian category with small direct sums, exact small products, and a generator. Then for an inverse sequence $S$ in $\AA$ the following statements are equivalent:
\begin{enumerate}
\item ${\lim}^1S=0;$
\item  $\lim {\rm Coker}(S\to \widehat S^\lambda)=0$ for any limit ordinal $\lambda$;
\item for a limit ordinal $\lambda,$ if cofinality of $\lambda$ is countable, then $S$ is $\lambda$-complete, if the cofinality of $\lambda$ is uncountable, then  $\lim {\rm Coker}(S\to \widehat S^{\lambda})=0.$
\end{enumerate}
\end{theorem*}
This theorem implies that if $S$ is $\lambda$-complete for any limit ordinal $\lambda,$ then $\lim^1 S=0.$ On the other hand, if $\lim^1 S=0$ and  $\lambda$ is a limit ordinal of countable cofinality, then $S$ is $\lambda$-complete. 

When studying the class of inverse sequences with trivial derived limits, it becomes clear that a more convenient class to investigate is the class of \emph{local inverse sequences} i.e. inverse sequences $L$ such that  $\lim L = \lim^1 L=0$. It turns out that any inverse sequence $S$ such that $\lim^1S=0$ can be uniquely decomposed into a short exact sequence $E\mono S \twoheadrightarrow L$, where $E$ is epimorphic, and $L$ is local (Cor. \ref{cor:unique_decomp}). Thus, many questions regarding inverse sequences with trivial derived limit are reduced to questions about local inverse sequences. 

The advantage of the class of local inverse sequences is that it is closed with respect to small limits. Moreover, under some additional assumptions on the abelian category $\AA$ there is a functor of localization of inverse sequences. Namely, if we, in addition to the above assumptions, assume that $\AA$ has a $\kappa$-compact generator for some regular cardinal $\kappa,$ then for any inverse sequence $S$ there is a universal morphism to a local inverse sequence
\[S \longrightarrow \mathcal{L}(S).\]
In other words, the subcategory of local inverse sequences is reflective (Prop.
\ref{prop:localization}).

The choice of the term ``local'' is related to the fact that for any category $\mathcal{C}$ and its morphism $\theta:c\to c'$, an object $l$ is called $\theta$-local if the map $\theta^*:\mathcal{C}(c',l)\to \mathcal{C}(c,l)$ is a bijection (see \cite{bousfield1975localization, farjoun2006cellular, libman2000cardinality, akhtiamov2021right}).  We show that an inverse sequence is local if and only if it is a $\theta$-local object in the category of inverse sequences with respect to some particular choice of $\theta$ (Prop. \ref{prop:1-I-local}).

A simplest example of a local inverse sequence is a \emph{null inverse sequence} i.e. an inverse sequence $N$ such that all the morphisms $N_{i+1}\to N_i$ are zero.  If there is a short exact sequence of inverse sequences $N\mono S'\epi S$, where $N$ is null, then $S'$ is called a null-extension of $S$. The second main result of our work is the following theorem (Th. \ref{th:description_of_local}). 

\begin{theorem*} Let $\AA$ be an abelian category with small direct sums, exact small products, and a generator. Then the class of local inverse sequences in $\AA$ is the least class of inverse sequences containing the trivial inverse sequence and closed with respect to small limits and null-extensions. 
\end{theorem*}
We draw an analogy between the category of groups and the category of inverse sequences. In this view, abelian groups are analogous to null inverse sequences, and central extensions are analogous to null-extensions. With this perspective, the theorem is analogous to Bousfield's description of the class of  $H\ZZ$-local groups as the least class containing the trivial group and closed with respect to small limits and central extensions \cite[Th. 3.10]{bousfield1977homological}. 

In the end of the paper, in order to illustrate the complexity of the class of local inverse sequences and emphasize the reasonableness of the statements of these theorems, we provide two types of examples of inverse sequences of abelian groups. Firstly, for each ordinal $\alpha$, we construct a local inverse sequence $S$ such that $I^\beta(S)\neq 0$ for any $\beta<\alpha$, but $I^\alpha(S)=0$  (Th. \ref{th:example}). Secondly, for each regular uncountable cardinal $\kappa$, we construct a local inverse sequence $S$ which is not $\kappa$-complete (Th. \ref{th:example2}). In particular for $\kappa=\aleph_1$ we obtain an example of a local inverse sequence which is $\lambda$-complete for all limit ordinals $\lambda$ except $\lambda=\aleph_1.$ 

In the course of our work, we raised the question: could it be the case that the condition $\lim^1 S=0$ is equivalent to the fact that $S$ is $\omega$-complete, and we do not need all the higher ordinals to formulate a necessarily and sufficient condition? We believe that this cannot be true, even in the category of abelian groups, but we have not been able to provide a counterexample. Therefore, we leave this question open for further investigation.

\begin{question*}
Is there an  $\omega$-complete inverse sequence of abelian groups with nonzero derived limit?
\end{question*}

We expect that, under some additional assumptions on the abelian category $\AA$, the analogy with Bousfield's theory of $H\ZZ$-localization of groups can be continued further. We think that there is a transfinite construction of the localization functor similar to the construction given in \cite{bousfield1977homological}, using the relative universal null-extensions similar to relative universal central extensions described in \cite{ivanov2018lengths} and \cite{farjoun2017relative}. However, we decided to leave this direction for further research.

\section*{Acknowledgements}
We are very grateful to Ekaterina Borodinova and Ioannis Emmanouil for useful discussions.

\section{Transfinite image filtration}

\subsection{Inverse sequences}

Further we will always assume that $\AA$ is an abelian category with a generator $G$ (i.e. the functor $\AA(G,-)$ is faithfull),  small direct sums (${\rm AB3}$) and exact small products (${\rm AB4}^*$). These assumptions imply that all small limits exists and left exact. An inverse sequence $S$ in $\AA$ is a couple consisting of two families $S=((S_i), (f_i))$ indexed by natural numbers $i\in \omega,$ where $S_i$ is an object in $\AA$ and  $f_i:S_{i+1}\to S_i$ is a morphism. One can say that inverse sequences are functors $S:\omega^{\scaleobj{0.7}{\rm op}}\to \AA.$ In particular, inverse sequences form an abelian category $\AA^{\omega^{\scaleobj{0.7}{\rm op}}}$ with small direct sums, exact small products (and a generator, Lemma \ref{lemma:a_generator}). Under these assumptions on $\AA$ for an inverse sequence $S$ in $\AA$ we have an exact sequence 
\begin{equation}\label{eq:lim_es}
0 \longrightarrow 
\lim S 
\longrightarrow 
\prod_{i} S_i 
\xrightarrow{\ 1-F\ } 
\prod_{i} S_i
\longrightarrow
{\lim}^1 S \longrightarrow 0,
\end{equation}
where ${\rm pr}_j F = f_j {\rm pr}_{j+1}$ (here ${\rm pr}_j: \prod_i S_i\to S_j$ denotes the canonical projection) and $\lim^n S=0$ for $n\geq 2$ \cite[Remark A.3.6.]{neeman2001triangulated}. We say that an inverse sequence is epimorphic, if the maps $f_i$ are epimorphisms. Roos proved that for an epimorphic inverse sequence $S$ we have ${\lim}^1 S=0$  \cite[Th. 3.1.]{roos2006derived}. We say that an inverse $S$ sequence is null, if  $f_i=0$ for each $i.$ It is easy to see that for a null inverse sequence $S$ we have $\lim S =\lim^1 S=0.$

Further we will take not only limits of inverse sequences, but also limits of some functors to the category of inverse sequences $F:J\to \AA^{\omega^{\scaleobj{0.7}{\rm op}}}.$ In this case we always use a subscript 
\begin{equation}
\lim_J : (\AA^{\omega^{\scaleobj{0.7}{\rm op}}})^J \longrightarrow \AA^{\omega^{\scaleobj{0.7}{\rm op}}}. 
\end{equation}
``$\lim$'' without any subscript always means a limit of an inverse sequence
\begin{equation}
\lim : \AA^{\omega^{\scaleobj{0.7}{\rm op}}  }\longrightarrow \AA.
\end{equation}

\subsection{Transfinite image filtration and completion} 
For an inverse sequence $S$ we denote by $S^{\rm sh}$ the shifted inverse sequence such that $S^{\rm sh}_i=S_{i+1}$ and $f^{S^{\rm sh}}_i=f^S_{i+1}.$ Then there is a morphism of inverse sequences
\begin{equation}
\tilde f:S^{\rm sh}\to S    
\end{equation}
defined by $\tilde f_i=f^S_i:S_{i+1}\to S_i.$ It is easy to check that the kernel and cokernel of $\tilde f$ are null inverse sequences. Therefore, $\tilde f$ induces  isomorphisms 
\begin{equation}\label{eq:shift}
\lim S^{\rm sh} \cong \lim S, \hspace{1cm}  {\lim}^1 S^{\rm sh} \cong {\lim}^1 S.  
\end{equation}
The inverse sequences defined by the image and the cokernel of $\tilde f$ are denoted by $I(S)$ and $S^1$ respectively. Then we have a short exact sequence
\begin{equation}
I(S) \mono S \epi S^1.
\end{equation}
It is easy to check that the morphism $S\epi S^1$ is a universal morphism from $S$ to a null inverse sequence. Note that $S^1=0$ if and only if $S$ is an epimorphic inverse sequence. 

Further for any ordinal number $\alpha$ we define $I^\alpha(S)$ such that $I^0(S)=S,$
\begin{equation}
I^{\alpha+1}(S)=I(I^\alpha(S)) \hspace{5mm} \text{and} \hspace{5mm}  I^\lambda(S)= \underset{\alpha<\lambda}\lim\: I^\alpha(S) 
\end{equation}
for a limit ordinal $\lambda.$ Since the functor of limit is left exact, the limit of monomorphisms is a monomorphism. So we get a transfinite tower of monomorphisms $I^\alpha(S)\mono S,$ which is called the transfinite image filtration of $S.$ For any ordinal $\alpha$ we denote by $S^\alpha$ the cokernel of the monomorphism $I^\alpha(S)\mono S.$ So there is a short exact sequence
\begin{equation}
I^\alpha(S) \mono S \epi S^\alpha.
\end{equation}
In proofs, when $S$ is fixed, we will simplify notation  $I^\alpha=I^\alpha(S).$  Using the snake lemma, it is easy to check that there is a short exact sequence
\begin{equation}\label{eq:ses_N}
(I^\alpha(S))^1 \mono S^{\alpha+1} \epi S^\alpha.    
\end{equation}
For any limit ordinal $\lambda$ the \emph{$\lambda$-completion} of an inverse sequence $S$ is defined as
\begin{equation}
\widehat S^\lambda = \lim_{\beta<\lambda} S^\beta.
\end{equation}
We also set $\widehat{S}^{\alpha+1}=S^\alpha$ for any ordinal $\alpha.$
The canonical projections $S\epi S^\beta$ define a natural map $S\to \widehat S^\lambda.$ 
$S$ is called \emph{$\lambda$-complete}, if the morphism $S\to \widehat S^\lambda$ is an epimorphism. Take a limit ordinal $\lambda.$ Since the functor $\lim_{\alpha<\lambda}$ is left exact, applying it to the short exact sequence $I^\alpha(S)\mono S \epi S^\alpha,$ we obtain  
$I^\lambda(S) = \Ker(S\to \widehat S^\lambda).$ It follows that for any limit ordinal $\lambda$  the morphism $S\to \widehat S^\lambda$ induces a monomorphism 
\begin{equation}\label{eq:S^lambda-mono}
S^\lambda \mono \widehat S^\lambda.
\end{equation}

\begin{proposition}\label{prop:limI^alpha}
For any ordinal $\alpha$ there are isomorphisms 
\begin{equation}\label{eq:lim_I^alpha}
\lim S^\alpha =\lim \widehat S^\alpha=0, \hspace{5mm} \lim I^\alpha(S) \cong \lim S,
\end{equation}
where the last isomorphism is induced by the monomorphism $I^\alpha(S)\mono S.$ Moreover, there is a short exact sequence 
\begin{equation}\label{eq:ses_lim^1}
{\lim}^1 I^\alpha(S) \mono {\lim}^1 S \epi {\lim}^1 S^\alpha.
\end{equation}
\end{proposition}
\begin{proof} 
The isomorphism $\lim I^\alpha(S) \cong \lim S$ follows from the equation $\lim S^\alpha=0,$ the short exact sequence $I^\alpha(S)\mono S \epi S^\alpha$ and the fact that the functor of limit is left exact. So it is sufficient to prove the equations $\lim S^\alpha = \lim \widehat S^\alpha=0.$
The proof is by transfinite induction. For $\alpha=0$ the statement is obvious. 

Assume that $\lim S^\alpha=\lim \widehat S^\alpha=0$ and prove $\lim S^{\alpha+1}=\lim \widehat S^{\alpha+1}=0.$ We have $\lim \widehat S^{\alpha+1}=0$ because $\widehat S^{\alpha+1}=S^\alpha.$ Using short exact sequence $(I^\alpha)^1\mono S^{\alpha+1} \epi S^{\alpha}$ \eqref{eq:ses_N}, the fact that $(I^\alpha)^1$ is null, and the left exactness of the functor of limit we obtain that $\lim S^{\alpha+1}=0$. 

Now assume that $\lambda$ is a limit ordinal and for any $\alpha<\lambda$ we have $\lim S^\alpha=\lim \widehat S^{\alpha}=0.$ Prove that $\lim S^\lambda=\lim \widehat S^{\lambda}=0.$ Since limits commute with limits, we obtain $\lim \widehat S^\lambda \cong \lim_{\alpha<\lambda} \lim S^\alpha = 0.$ The equation $\lim S^\lambda=0$ follows from the embedding $S^\lambda \mono \widehat S^\lambda$ \eqref{eq:S^lambda-mono} and left exactness of the limit. 
\end{proof}

\subsection{Length of the transfinite image filtration}

The \emph{length of the transfinite image filtration} of $S$ is the   the least ordinal ${\rm len}(S)$ such that for any $\alpha>{\rm len}(S)$ the monomorphism $I^\alpha (S)\mono I^{{\rm len}(S)}(S)$ is an isomorphism. Further in Proposition \ref{prop:mu} we will show that it is well defined for any $S$.  

\begin{proposition}\label{prop:mu}
For an inverse sequence $S$ the ordinal ${\rm len}(S)$ is well defined. Moreover, $I^{{\rm len}(S)}(S)$ is an epimorphic inverse sequence, the canonical morphisms $I^{{\rm len}(S)}(S) \mono S\epi S^{{\rm len}(S)}$ induce  isomorphisms 
\begin{equation}
{\lim}^1 S \cong {\lim}^1 S^{{\rm len}(S)}, \hspace{1cm} \lim I^{{\rm len}(S)}(S)\cong \lim S   ,
\end{equation}  
and there is an isomorphism 
\begin{equation}
I^{{\rm len}(S)}(S)\cong \Im( \lim S\to S),
\end{equation}
where $\lim S$ is treated as a constant inverse sequence. 
\end{proposition}
\begin{proof} 
Since $\AA$ has a generator, it is well-powered \cite[Prop. 3.35]{freyd1964abelian}. It follows that the transfinite decreasing sequence of subobjects $I^\alpha \mono S$ stabilises, and there is an ordinal $\mu$ such that for any $\alpha>\mu$ the monomorphism  $I^\alpha\mono I^\mu$ is an isomorphism. 
Therefore, we can take the least ordinal with this property and denote it by ${\rm len}(S).$ 
This property implies that $I(I^{\mu})\to I^{\mu}$ is an isomorphism. Hence $I^{\mu}$ is an epimorphic inverse sequence. 
The result of Roos \cite[Th. 3.1]{roos2006derived} implies that ${\lim}^1 I^{\mu}=0.$ Then the first isomorphism follows from the short exact sequence \eqref{eq:ses_lim^1}, and the second one follows from \eqref{eq:lim_I^alpha}. Since $I^\mu$ is epimorphic, we obtain that the morphism $\lim I^\mu \to I^\mu$ is an epimorphism. Using the commutativity of the diagram
\begin{equation}
\begin{tikzcd}
\lim I^\mu \ar[twoheadrightarrow]{r} \ar{d}{\cong} & I^\mu \ar[rightarrowtail]{d} \\
\lim S \ar{r} & S
\end{tikzcd}
\end{equation}
we obtain that $I^\mu\cong \Im( \lim S\to S ).$
\end{proof}

\begin{corollary}\label{cor:limS=0}
For an inverse sequence $S$ the following statements are equivalent 
\begin{enumerate}
    \item $\lim S=0;$
    \item $I^{{\rm len}(S)}(S)=0.$
\end{enumerate}
\end{corollary}

\begin{corollary} \label{cor:epimorphic}
If $S$ is an epimorphic inverse sequence and $\lim S=0,$ then $S=0.$ 
\end{corollary}

\begin{corollary}\label{cor:unique_decomp}
For an inverse sequence $S$ there exists a unique (up to isomorphism) short exact sequence
\begin{equation}
E \mono S \epi S'
\end{equation}
such that $E$ is epimorphic and $\lim S'=0.$ Moreover, for such an exact sequence we have $\lim^1 S \cong \lim^1 S'.$
\end{corollary}
\begin{proof}
Propositions \ref{prop:mu} and \ref{prop:limI^alpha} imply that $I^{{\rm len}(S)}(S) \mono S \epi S^{{\rm len}(S)}$ satisfies this property. Assume that $E \mono S \epi S'$ is a such a short exact sequence and prove that it is isomorphic to $I^{{\rm len}(S)}(S) \mono S \epi S^{{\rm len}(S)}.$ 
Consider an ordinal number $\alpha$ such that $\alpha\geq {\rm len}(S), {\rm len}(S').$ Then $I^\alpha(E)=E, I^{\alpha}(S)=I^{{\rm len}(S)}(S)$ and $I^{\alpha}(S')=I^{{\rm len}(S')}(S').$ Corollary \ref{cor:limS=0} implies that $I^{\alpha}(S')=0.$ Therefore the composition $I^\alpha(S) \to S \to S'$ is trivial. Therefore, there is a morphism $I^\alpha(S) \mono E.$ The assumption $\lim S'=0$ implies that the morphism 
$\lim E \to \lim S$ is an isomorphism. By Proposition \ref{prop:mu} we have that $\lim I^\alpha(S) \to \lim S$ is an isomorphism. Therefore the map $\lim I^\alpha(S) \to \lim E$ is an isomorphism. Since $E$ and $I^\alpha(S)$ are epimorphic, we obtain that ${\rm Coker}(I^\alpha(S)\mono E)$ is epimorphic and $\lim {\rm Coker}(I^\alpha(S)\mono E)=0.$ Then Corollary \ref{cor:epimorphic} implies that ${\rm Coker}(I^\alpha(S)\mono E)=0$ and the morphism $I^\alpha(S)\mono E$ is an isomorphism. 
\end{proof}

\section{Local inverse sequences}
An inverse sequence $S$ is called \emph{local} if $\lim S={\lim}^1S=0.$ 
The exact sequence \eqref{eq:lim_es} implies that $S$ is local if and only if the morphism $1-F : \prod S_i \to \prod S_i$ is an isomorphism. 
It is easy to see that the category of local inverse sequences is a Serre subcategory of the category of all inverse sequences. In particular, the class of local inverse sequences is closed with respect to extensions. Null inverse sequences are local. An extension of two null inverse sequences is not necessarily null but it is still an example of a local inverse sequence.

Next, we explain the choice of the term ``local'' and prove some properties of local inverse sequences. For a general category $\CC$ and a morphism $\theta:c'\to c$ an object $l$ of $\CC$ is called \emph{$\theta$-local}, if the map $\theta^*:\CC(c,l)\to \CC(c',l)$ is a bijection. The class of $\theta$-local objects is closed with respect to small limits \cite[\S 1.5]{libman2000cardinality}. Further we show that local inverse sequences are $\theta$-local objects with respect to particular morphism $\theta$ in $\AA^{\omega^{\scaleobj{0.7}{\rm op}}}.$ As a corollary we obtain that the class of local inverse sequences is closed with respect to small limits and extensions.

For an object $A$ in $\AA$ and a natural number $n$ we denote by $A(n)$ the inverse sequence such that $A(n)_i=A$ for $i\leq n,$  $A(n)_i=0$ for $i>n,$ and $f_i=1_A$ for $i<n.$ We denote by $\iota(n):A(n)\to A(n+1)$ the morphism of inverse sequences such that $\iota(n)_i=1_A$ for $i\leq n$ and $\iota(n)_i=0$ for $i>n.$ 
\begin{equation}
\begin{tikzcd}
A(n)\ar[d,"\iota(n)"] &  A \ar[d,"1_A"] & \dots \ar[l,"1_A"'] & A \ar[l,"1_A"'] \ar[d,"1_A"] & 0 \ar[l] \ar[d] & 0 \ar[d] \ar[l] & \dots \ar[l] \\
A(n+1) & A & \dots \ar[l,"1_A"'] & A \ar[l,"1_A"'] & A \ar[l,"1_A"] & 0 \ar[l] & \dots \ar[l]
\end{tikzcd}   
\end{equation}

\begin{lemma}\label{lemma:A(n)}
For any object $A$ and any inverse sequence $S$ there is an natural (adjunction) isomorphism  
\begin{equation}\label{eq:iso_A(n)}
\AA^{\omega^{\scaleobj{0.7}{\rm op}}}(A(n),S) \cong \AA(A,S_n), \hspace{1cm} \varphi \mapsto \varphi_n.  
\end{equation}
Moreover, the diagram 
\begin{equation}
\begin{tikzcd}
\AA^{\omega^{\scaleobj{0.7}{\rm op}}}(A(n+1),S)
\ar[d,"\iota(n)^*"]
\ar[r,"\cong"] 
& \AA(A,S_{n+1}) 
\ar[d,"(f_n)_*"] 
\\
\AA^{\omega^{\scaleobj{0.7}{\rm op}}}(A(n),S) 
\ar[r,"\cong"] 
& \AA(A,S_{n})
\end{tikzcd}
\end{equation}
is commutative. 
\end{lemma}
\begin{proof}
Straightforward. 
\end{proof}

Further we fix a generator $G$ of $\AA$ and consider an inverse sequence defined by
\begin{equation}
\tilde G := \bigoplus_{i<\omega} G(i).
\end{equation}

\begin{lemma}\label{lemma:a_generator}
$\tilde G$ is a generator of $ \AA^{\omega^{\scaleobj{0.7}{\rm op}}}$ and there is an isomorphism 
\begin{equation}\label{eq:prod}
\AA^{\omega^{\scaleobj{0.7}{\rm op}}}(\tilde G, S) \cong {\prod}_i\: \AA(G,S_i).
\end{equation}
\end{lemma}
\begin{proof}
The isomorphism \eqref{eq:prod} 
follows from the isomorphism $\AA^{\omega^{\scaleobj{0.7}{\rm op}}}({\bigoplus}_i\: G(i), S)\cong  \prod_i \AA^{\omega^{\scaleobj{0.7}{\rm op}}}(G(i), S)$ and the isomorphism \eqref{eq:iso_A(n)}. The fact that $\tilde G$ is a generator follows from the fact that $G$ is a generator, and the isomorphism \eqref{eq:prod}.
\end{proof}

Consider a morphism of inverse sequences 
\begin{equation}
1 - I : \tilde G \longrightarrow \tilde G, 
\end{equation}
where $I{\rm em}_n = {\rm em}_{n+1} \iota(n)$ and ${\rm em}_n: G(n)\to \tilde G$ is the canonical embedding. 
\begin{proposition}\label{prop:1-I-local}
An inverse sequence $S$ is local if and only if it is a $(1-I)$-local object of $\AA^{\omega^{\scaleobj{0.7}{\rm op}}}.$
\end{proposition}
\begin{proof}
Lemma \ref{lemma:A(n)} implies that there is an isomorphism 
\begin{equation}
\AA^{\omega^{\scaleobj{0.7}{\rm op}}}\left(\tilde G, S\right)  \cong {\prod}_i\: \AA(G, S_i) \cong \AA\left(G,{\prod}_i S_i\right)
\end{equation}
and the diagram 
\begin{equation}
\begin{tikzcd}
\AA^{\omega^{\scaleobj{0.7}{\rm op}}} \left(\tilde G, S\right) 
\ar[r,"\cong"] \ar[d,"(1-I)^*"]
& 
\AA\left(G, \prod_i S_i \right) 
\ar[d,"(1-F)_*"]
\\
\AA^{\omega^{\scaleobj{0.7}{\rm op}}} \left(\tilde G, S\right) 
\ar[r,"\cong"]\ar[r,"\cong"] 
& 
\AA\left(G, \prod_i S_i \right)
\end{tikzcd}
\end{equation}
is commutative. Therefore, $S$ is $(1-I)$-local if and only if the map $(1-F)_*$ is an isomorphism. Since $G$ is a generator, the functor $\AA(G,-)$ is faithfull. A faithfull functor reflects monomorphisms and epimorphisms \cite[Th.7.1]{mitchell1965theory}. Therefore $\AA(G,-)$ reflects isomorphisms.  It follows that $(1-F)_*$ is an isomorphism if and only if $1-F$ is an isomorphism, which is equivalent to the fact that $S$ is local. 
\end{proof}

\begin{corollary}\label{cor:closed_under_limits}
The class of local inverse sequences is closed with respect to small limits.
\end{corollary}

Let $\kappa$ be a regular cardinal. We say that a poset is $\kappa$-directed if any its subset of cardinality $<\kappa$ has an upper bound.  An object $c$ of a category $\CC$ is called $\kappa$-compact or ($\kappa$-presentable), if the hom-functor $\CC(c,-)$ commutes with colimits over $\kappa$-directed posets. Note that if $\kappa<\kappa',$ then a $\kappa$-compact object is also $\kappa'$-compact.

\begin{proposition}\label{prop:localization} Assume that $\AA$ has a $\kappa$-compact generator $G$ for some regular cardinal $\kappa.$
Then for any inverse sequence $S$ in $\AA$ there exists a universal initial morphism to a local inverse sequence 
\begin{equation}
S \longrightarrow \mathcal{L}(S).
\end{equation}
In other words, the full subcategory of local inverse sequences is a reflective subcategory of $\AA^{\omega^{\scaleobj{0.7}{\rm op}}}$. 
\end{proposition}
\begin{proof}  
Without loss of generality we can assume that  $\kappa$ is uncountable. 
Then any countable direct sum of $\kappa$-compact objects is $\kappa$-compact \cite[Prop.1.16]{adamek1994locally}. Since $G$ is $\kappa$-compact and colimits of inverse sequences are computed level-wise, using Lemma \ref{lemma:A(n)} we obtain that $G(i)$ is also $\kappa$-compact. Therefore $\tilde G=\bigoplus_i G(i)$ is $\kappa$-compact as well. Hence the assertion follows from Proposition \ref{prop:1-I-local} combined with the result of Casacuberta, Peschke and Pfenniger 
\cite[Cor. 1.7]{casacuberta1991orthogonal}.
\end{proof}

\section{Transfinite Mittag-Leffler condition}

\begin{theorem}\label{th:gen_ML}
Let $S$ be an inverse sequence in $\AA.$ Then the following statements are equivalent:
\begin{enumerate}
\item ${\lim}^1S=0;$
\item  $\lim {\rm Coker}(S\to \widehat S^\lambda)=0$ for any limit ordinal $\lambda$;
\item for a limit ordinal $\lambda,$ if cofinality of $\lambda$ is countable, then $S$ is $\lambda$-complete, if the cofinality of $\lambda$ is uncountable, then  $\lim {\rm Coker}(S\to \widehat S^{\lambda})=0.$
\end{enumerate}
\end{theorem}
\begin{proof} 
For the sake of convenience we set $C^\lambda={\rm Coker}(S\to \widehat S^\lambda).$ 

$(1)\Rightarrow (2).$ Using the short exact sequence $S^\lambda\mono \widehat S^\lambda \epi C^\lambda$ (see \eqref{eq:S^lambda-mono}), we obtain that there is an exact sequence $ \lim \widehat S^\lambda \to \lim C^\lambda \to \lim^1 S^\lambda.$ So it is sufficient to check that $\lim \widehat S^\lambda =0$ and $\lim^1 S^\lambda =0.$ The first equality follows from Proposition \ref{prop:limI^alpha}. The second equality follows from the fact that there is an epimorphism $S\epi S^\lambda$ and  $\lim^1 S=0$ (see \eqref{eq:ses_lim^1}). 

$(2)\Rightarrow (1).$ By Proposition \ref{prop:mu} it is sufficient to prove that $S^\alpha$ is local for any $\alpha.$ Prove it by the transfinite induction. For $\alpha=0$ it is obvious. Assume that $S^\alpha$ is local, and prove that $S^{\alpha+1}$ is local. It follows from the short exact sequence $(I^\alpha)^1\mono S^{\alpha+1}\epi S^\alpha$ and the fact that  $(I^\alpha)^1$ is null.

Now assume that $\lambda$ is a limit ordinal and $S^\alpha$ is local for any $\alpha<\lambda,$ and prove that $S^\lambda$ is local. Since local inverse sequences are closed with respect to small limits (Corollary \ref{cor:closed_under_limits}), $\widehat S^\lambda$ is local. The short exact sequence $S^\lambda \mono \widehat S^\lambda \epi C^\lambda$ implies that there is an exact sequence $\lim C^\lambda \to  \lim^1 S^\lambda \to \lim^1 \widehat S^\lambda.$ By the assumption $\lim C^\lambda=0.$ Since $\widehat S^\lambda$ is local, $\lim^1 \widehat S^\lambda=0.$ Therefore $\lim^1 S^\lambda=0.$ By Proposition \ref{prop:limI^alpha} we have $\lim S^\lambda=0.$ Therefore $S^\lambda$ is local.

$(3)\Rightarrow (2).$ Obvious. 

$(1)\&(2) \Rightarrow (3).$ Take an ordinal $\lambda$ with countable cofinality. Since a shifted inverse sequence has the same $\lim$ and $\lim^1$ \eqref{eq:shift}, it is sufficient to prove that $S_0\to \widehat S^\lambda_0$ is an epimorphism. Let $\alpha_i$ be a strictly increasing sequence of ordinals that tends to $\lambda.$ 
Then the sequence $\alpha_i+i$ is  also strictly increasing and tends to $\lambda.$ Therefore $\widehat S^\lambda_0 = \lim_i S^{\alpha_i+i}_0.$ The short exact sequence $I^{\alpha_i+i}_0 \mono S_0 \epi S^{\alpha_i+i}_0$ implies that there is an exact sequence $S_0 \to \widehat S^\lambda_0 \to \lim^1_i I^{\alpha_i+i}_0.$ Therefore, it is sufficient to prove that $\lim^1_i I^{\alpha_i+i}_0=0.$ There is an epimorphism $I^{\alpha_i}_i \epi I^{\alpha_i+i}_0.$ Therefore, it is sufficient to prove that $\lim^1_i I^{\alpha_i}_i=0.$ Consider the short exact sequence
$I_i^{\alpha_i}\mono S_i \epi S_i^{\alpha_i}.$ Then we have an exact sequence $ \lim_i S_i^{\alpha_i} \to \lim^1_i I_i^{\alpha_i} \to \lim^1 S.$ By the assumption $\lim^1 S =0.$ Then it is sufficient to show that $\lim_i S_i^{\alpha_i}=0.$ Since the diagonal $\{(i,i)\mid i\in \omega \}$ is cofinal in $\omega\times \omega,$ we have $\lim_i S_i^{\alpha_i} = \lim_{(i,j)\in \omega\times \omega} S_i^{\alpha_j} = \lim_i \lim_j S_i^{\alpha_j} = \lim_i \widehat S_i^\lambda = \lim \widehat S^\lambda.$ Then the assertion follows from Proposition \ref{prop:limI^alpha}.
\end{proof}

\begin{corollary}
If $S$ is $\lambda$-complete for any limit ordinal $\lambda,$ then $\lim^1S=0.$ 
\end{corollary}

\begin{corollary}
If an inverse sequence $S$ (in an abelian category $\AA$ with a generator, small direct sums and exact small products) satisfies the Mittag-Leffler condition, then $\lim^1 S=0.$
\end{corollary}

\begin{corollary}
If $\lim^1 S=0$ and $\lambda$ is a limit ordinal of countable cofinality, then $S$ is $\lambda$-complete. 
\end{corollary}

\section{A description of the class of local inverse sequences}

If we have a short exact sequence of inverse sequences
$N \mono S' \epi S$
such that $N$ is null, we say that $S'$ is a \emph{null extension} of $S.$ We think about null extensions of inverse sequences as analogues of central extensions of groups.

\begin{theorem}\label{th:description_of_local}
The class of local inverse sequences in $\AA$ is the least class containing the zero inverse sequence and closed with respect to small limits and null-extensions. 
\end{theorem}
\begin{proof}
Corollary \ref{cor:closed_under_limits} says that the class of local inverse sequences is closed under limits. It is obviously closed with respect to null extensions. Let us prove that it is the least class satisfying these properties. For this we fix a class satisfying this properties $\mathcal{L}$ and prove that any local inverse sequence $S$ is in $\mathcal{L}.$ Note that, since $\mathcal{L}$ is closed with respect to null extensions and contains zero inverse sequence, all null inverse sequences are in $\mathcal{L}.$ Also note that an isomorphism $S\cong S'$ can be treated as an extension with zero kernel. So if $S\cong S'$ and $S\in \mathcal{L},$ then $S'\in \mathcal{L}.$

Assume that $S$ is a local inverse sequence and prove that $S\in \mathcal{L}$. By Corollary \ref{cor:limS=0} we have $S=S^{{\rm len}(S)}.$ Then it is sufficient to prove by transfinite induction that for any ordinal $\alpha$ we have $S^\alpha\in \mathcal{L}.$ For $\alpha=0$ it is obvious. If $S^\alpha\in \mathcal{L}$ then the short exact sequence \eqref{eq:ses_N} implies that $S^{\alpha+1}\in \mathcal{L}.$

Let $\lambda$ is a limit ordinal and assume that for any $\alpha<\lambda$ we have $S^\alpha\in \mathcal{L}.$ The rest of the proof is devoted to the proof that $S^\lambda\in \mathcal{L}.$ Since $\mathcal{L}$ is closed under small limits, we get $\widehat S^\lambda \in \mathcal{L}.$ Recall that we have a monomorphism $S^\lambda\mono \widehat S^\lambda$ and set $C^\lambda:=\Coker(S^\lambda\mono \widehat S^\lambda).$ 
For any ordinal $\beta$ we define a decomposition of this monomorphism into two monomorphisms 
\begin{equation}
S^\lambda \mono J^\beta \mono \widehat S^\lambda,
\end{equation}
where $J^\beta$ is defined as the pullback
\begin{equation}\label{eq:pullback_J}
\begin{tikzcd}
J^\beta \ar[d,twoheadrightarrow] \ar[r,rightarrowtail]
& 
\widehat S^\lambda 
\ar[d,twoheadrightarrow]
\\
I^\beta(C^\lambda) \ar[r,rightarrowtail]
& 
C^\lambda.
\end{tikzcd}
\end{equation}
By Theorem \ref{th:gen_ML} we have $\lim C^\lambda=0.$ By Corollary \ref{cor:limS=0} we obtain $I^{{\rm len}(C^\lambda)}(C^\lambda)=0.$ Therefore 
\begin{equation}
S^\lambda = J^{{\rm len}(C^\lambda)}.    
\end{equation}
It follows that in order to complete the proof it is sufficient  to check that $J^\beta\in \mathcal{L}$ for any $\beta.$ Let us do it.

For $\beta=0$ we have $J^0=\widehat S^\lambda\in \mathcal{L}.$  Assume that  $J^\beta \in \mathcal{L}$ and prove that $J^{\beta+1}\in \mathcal{L}.$ 
Since \eqref{eq:pullback_J} is a pullback, the kernels of the vertical arrows are isomorphic 
\begin{equation}\label{eq:ker_J}
\Ker(J^\beta \epi I^\beta(C^\lambda)) \cong S^\lambda.
\end{equation}
It follows that $\Ker(J^\beta\to I^\beta(C^\lambda))\cong \Ker(J^{\beta+1} \to I^{\beta+1}(C^\lambda)).$ Therefore, the snake lemma implies that the right hand vertical arrow in the following diagram is an isomorphism
\begin{equation}
\begin{tikzcd}
J^{\beta+1} \ar[d,twoheadrightarrow] \ar[r,rightarrowtail]
& 
J^\beta 
\ar[d,twoheadrightarrow] \ar[r,twoheadrightarrow]
&
\Coker(J^{\beta+1}\to J^\beta)\ar[d,"\cong"]
\\
I^{\beta+1}(C^\lambda) \ar[r,rightarrowtail]
& 
I^\beta(C^\lambda)
\ar[r,twoheadrightarrow]
& 
(I^\beta(C^\lambda))^1.
\end{tikzcd}
\end{equation}
Since  $(I^\beta(C^\lambda))^1$ is null, we have $(I^\beta(C^\lambda))^1\in \mathcal{L}.$ By the assumption $J^\beta\in \mathcal{L}.$ Therefore $J^{\beta+1}$ is a kernel of a morphism between two objects of $\mathcal{L}.$ Since the kernel is a limit, we obtain $J^{\beta+1}\in \mathcal{L}.$ 

Now assume that for a limit ordinal $\mu$ and any $\beta<\mu$ we have $J^\beta \in \mathcal{L}.$ Since limits commute with limits, the equality  $I^\mu(C^\lambda)=\lim_{\beta<\mu} I^{\beta}(C^\lambda)$ implies the isomorphism $J^\mu \cong \lim_{\beta<\mu} J^\beta.$ Therefore $J^\mu\in \mathcal{L}.$ So we proved that $J^\beta \in \mathcal{L}$ for any ordinal $\beta.$
\end{proof}

\section{Examples of inverse sequences of abelian groups}

\subsection{Inverse sequences defined by one abelian group}

Further we fix a prime $p.$
For an abelian group $A$ and an ordinal $\alpha$ we denote by $p^\alpha A$ a subgroup of $A$ defined so that $p^0A=A,$ 
$p^{\alpha+1}A=p\cdot (p^\alpha A)$ and $p^\lambda A = \bigcap_{\alpha<\lambda} p^\alpha A$ for a limit ordinal $\lambda.$ If $A$ is a $p$-group, the group $p^{\omega \alpha}A$ is known as the $\alpha$-th Ulm subgroup. The least $\alpha$ such that $p^\alpha A=p^{\alpha+1}A$ is called the $p$-length of $A$ and denoted by ${\rm len}_p(A)$. We will also use the notation $p^\infty A=p^{{\rm len}_p(A)}A.$ It is known that for any ordinal $\alpha$ there exists an abelian $p$-group $A,$ whose length is equal to $\alpha$ \cite[\S 11, Exercise 43]{kaplansky2018infinite}, \cite[Ch.11, Example 3.2]{fuchs2015abelian}
\begin{equation}
{\rm len}_p(A) = \alpha. 
\end{equation}
Note that for any homomorphism $f:A\to B$ there is an inclusion $f(p^\alpha A)\subseteq p^\alpha B.$ However, in general $f(p^\alpha A)\neq p^\alpha B,$ even if $f$ is an epimorphism.

Consider an inverse sequence $S(A)$ such that $S(A)_i=A$ and $f_i(a)=pa.$ 
\begin{equation}
S(A): \hspace{1cm} A \overset{p\cdot }\longleftarrow A \overset{p\cdot}\longleftarrow A \overset{p\cdot }\longleftarrow \dots
\end{equation}

It is easy to see that 
\begin{equation}
I^\alpha(S(A))=S(p^\alpha A).    
\end{equation}
Therefore the length of the image filtration is equal to the $p$-length of $A.$
\begin{equation}
{\rm len}(S(A)) = {\rm len}_p(A).
\end{equation}
It follows that for any ordinal $\alpha$ there exists an inverse sequence of abelian groups $S$ such that ${\rm len}(S)=\alpha.$ In the next subsection we will construct an explicit abelian group $A$ such that ${\rm len}(S(A))=\alpha$ and $S(A)$ is local. Corollary \ref{cor:limS=0} implies that 
\begin{equation}\label{eq:limS(A)}
\lim S(A) =0 \hspace{5mm} \Leftrightarrow \hspace{5mm} p^\infty A=0.
\end{equation}
Further in this section we construct some examples of abelian groups $A$ such that inverse sequences $S(A)$ satisfy certain properties that confirm the reasonableness of Theorem \ref{th:gen_ML}.

\subsection{A local inverse sequence with a long image filtration}

In this section for any $\alpha$ we will construct an abelian group $E_\alpha$ such that $S(E_\alpha)$ is local and ${\rm len}(S(E_\alpha))=\alpha.$ The group $E_\alpha$ is a variant of Walker's group \cite{walker1974groups}, \cite[Ch.11, Example 3.2]{fuchs2015abelian}. Further in this subsection we fix an ordinal $\alpha.$

Denote by $\alpha^\diamond$ the set of all finite increasing sequences of ordinals $(\alpha_1,\dots,\alpha_n)$ such that  
$\alpha_1< \dots <\alpha_n<\alpha,$ $n\geq 1.$ We endow $\alpha^\diamond$ by the deg-lex order: $(\alpha_1,\dots,\alpha_n)< (\alpha'_1,\dots,\alpha'_{n'})$ if and only if either $n<n',$ or $n=n'$ and there exists $1\leq m\leq n$  such that $\alpha_m<\alpha'_m$ and $\alpha_i=\alpha'_i$ for any $1\leq i<m.$  It easy to check that it is a well order.

Consider the direct product and the direct sum of the group of $p$-adic integers 
$\ZZ_p$ indexed by $\alpha^\diamond$
\begin{equation}
P_\alpha := \ZZ_p^{\prod \alpha^\diamond}, \hspace{1cm} P'_\alpha := \ZZ_p^{\oplus \alpha^\diamond}.
\end{equation} 
We will treat the abelian groups $P_\alpha$ and $P'_\alpha$ as $\ZZ_p$-modules. Note that $P'_\alpha$ is a free $\ZZ_p$-module.

We denote by $(e_\sigma)_{\sigma\in \alpha^\diamond}$ the standard basis of $P'_\alpha$ over $\ZZ_p.$ 
Consider a $\ZZ_p$-submodule  $R_\alpha\subseteq P'_\alpha$ generated by the  elements of the form
\begin{equation}\label{eq:R}
r_{\alpha_1}:=pe_{\alpha_1} , \hspace{1cm}  r_{\alpha_1,\dots,\alpha_n}:=p  e_{\alpha_1,\dots,\alpha_n} - e_{\alpha_2,\dots, \alpha_n}, \hspace{1cm} n\geq 2,
\end{equation}
and set 
\begin{equation}\label{eq:A(alpha)}
D_\alpha := P_\alpha/R_\alpha, \hspace{5mm} 
D'_\alpha := P'_\alpha/R_\alpha,
\hspace{5mm} 
E_\alpha := D_\alpha/p^\alpha D_\alpha.
\end{equation}

\begin{theorem}\label{th:example} 
For any ordinal $\alpha$ the inverse sequence $S(E_\alpha)$ is local and
\begin{equation}
{\rm len}(S(E_\alpha))=\alpha.    
\end{equation}
\end{theorem}
The rest of this subsection is devoted to the proof of the theorem. We will need some additional constructions and lemmas for this.

For $t\in P_\alpha$ the elements ${\rm pr}_\sigma(t),$ where $\sigma\in \alpha^\diamond,$ will be called coordinates of $t.$ If $t\in P'_\alpha,$ the leading index  ${\rm li}(t)$ is the 
maximal $\sigma$ (with respect to the deg-lex order) such that 
${\rm pr}_\sigma(t)\ne 0.$ The corresponding coefficient is called the leading coordinate of $t.$

For any ordinal $\beta$ we denote by $P'_{[\beta,\alpha)}$ the free $\ZZ_p$-submodule of $P'_\alpha$ generated by the elements $e_{\alpha_1,\dots,\alpha_n}$ such that $\alpha_1\geq \beta.$ In other words, $ P'_{[\beta,\alpha)}$ is the submodule of $P'_\alpha$ consisting of elements, whose coordinates with indexes $(\alpha_1,\dots,\alpha_n)$ such that $\alpha_1<\beta$ are trivial.

\begin{lemma}\label{lemma:R}
The leading coordinate of an  element of $R_\alpha\setminus\{0\}$ is in $p\ZZ_p\setminus\{0\}.$
\end{lemma}
\begin{proof}
Any element of $r\in R_\alpha\setminus \{0\}$  can be presented as $\sum_{i=1}^n x_ir_{\sigma_i}$ such that $n\geq 1, x_i\in \ZZ_p\setminus\{0\}$ and $\sigma_1<\dots<\sigma_n.$ Then ${\rm li}(r)=\sigma_n$ and the leading coefficient is equal to $px_n.$
\end{proof}

\begin{lemma}
\label{lemma:presentation}
Any element $t\in P'_\alpha$ can be uniquely presented as 
\begin{equation}
t=t_0+r    
\end{equation}
such that $r\in R_\alpha, t_0\in P'_\alpha$ and all coordinates of $t_0$ are in $\{0,\dots,p-1\}.$ Moreover, for this presentation we have  ${\rm li}(t_0), {\rm li}(r)\leq {\rm li}(t)$ and, if $t\in P'_{[\beta,\alpha)},$ then $t_0\in P'_{[\beta,\alpha)}.$
\end{lemma}
\begin{proof}
First we prove the existence of such presentation $t=t_0+r$ such that $r\in R_\alpha, t_0\in P'_\alpha$ and all coordinates of $t_0$ are in $\{0,\dots,p-1\}$ and ${\rm li}(t_0), {\rm li}(r)\leq {\rm li}(t)$. 
Assume the contrary, that there exists $t\in P'_\alpha$ such that there is no such a presentation. Chose such $t,$ for which there is no such a representation, so that ${\rm li}(t)$ is the least possible (it is possible because 
$\alpha^\diamond$ is well ordered). 
Then $t=xe_{{\rm li}(t)} + \tilde t,$ 
where $x\in \ZZ_p\setminus\{0\}$ and either $\tilde t=0,$ 
or $\tilde t\in P'_\alpha$ such that ${\rm li}(\tilde t)<{\rm li}(t).$ Note that if $t\in P'_{[\beta,\alpha)},$ then $\tilde t\in P'_{[\beta,\alpha)}.$ 
Let's present $x$ as $x=x_0+ p\tilde x,$ 
where $x_0\in \{ 0,\dots,p-1 \}$ and $\tilde x\in \ZZ_p.$ 
Then $t=x_0e_{{\rm li}(t)} + p\tilde x e_{{\rm li}(t)}  + \tilde t.$ If ${\rm li}(t)=(\alpha_1),$ then we set $r:=p\tilde x e_{\alpha_1}\in R_\alpha$ and $\widehat t:=\tilde t.$  If ${\rm li}(t)=(\alpha_1,\dots,\alpha_n)$ for $n\geq 2,$ then we set  $r:=p\tilde x e_{\alpha_1,\dots,\alpha_n} - \tilde x e_{\alpha_2,\dots,\alpha_n}\in R_\alpha,$ and $\widehat t:=\tilde t+\tilde x e_{\alpha_2,\dots,\alpha_n}.$ 
In both cases we obtain $t=x_0e_{{\rm li}(t)}+\widehat t+r,$ where $r\in R_\alpha$ and $\widehat t\in P'_\alpha$ such that ${\rm li}(\widehat t)<{\rm li}(t)$ and ${\rm li}(r)\leq {\rm li}(t).$ Note that, if $t\in P'_{[\beta,\alpha)},$ then $\widehat t\in P'_{[\beta,\alpha)}.$ 
By the assumption we can present $\widehat t$ as $\widehat t=\widehat t_0+\widehat r,$ where $\widehat r\in R_\alpha,$ $\widehat t_0\in P'_\alpha,$ ${\rm li}(\widehat r),{\rm li}(\widehat t_0)\leq {\rm li}(\widehat t)<{\rm li}(t)$ and all coordinates of $\widehat t_0$ are from $\{0,\dots,p-1\}.$ Therefore $t=(x_0e_{{\rm li}(t)}+\widehat t_0)+(r+\widehat r).$ We claim that this presentation satisfies all the assumptions. Indeed, since ${\rm li}(\widehat t_0)<{\rm li}(t),$ we get that the coordinates of $x_0e_{{\rm li}(t)}+\widehat t_0$ are in $\{0,\dots, p-1  \}.$ Other conditions are obvious. This makes a contradiction. So we proved the existence. 

Let us prove the uniqueness. If we have two presentations $t_0+r=t'_0+r'$  that $r,r'\in R_\alpha$ and all coordinates of $t_0,t_0'$ are in $\{0,\dots, p-1\},$ then all coordinates of $t_0-t'_0$ are in $\{-(p-1),\dots,p-1\}.$ Using Lemma \ref{lemma:R} and the fact that the sets $\{-(p-1),\dots,p-1\}$ and $p\ZZ_p\setminus \{0\}$ don't intersect, we obtain $t_0=t'_0.$ 
\end{proof}

\begin{lemma}\label{lemma:p^betaD'}
For any ordinal $\beta$ we have
\begin{equation}
p^\beta D'_\alpha =(p^\beta D_\alpha)\cap D'_\alpha= {\rm Im}(P'_{[\beta,\alpha)} \to D'_\alpha).
\end{equation}
\end{lemma}
\begin{proof}  First we prove that for any ordinal $\beta$ we have
\begin{equation}\label{eq:P'_geq1}
 pP'_{[\beta,\alpha)} +R_\alpha =  P'_{[\beta+1,\alpha)}+R_\alpha.
\end{equation}
In order to prove this equation, we need to check two inclusions $pP'_{[\beta,\alpha)} \subseteq P'_{[\beta+1,\alpha)}+R_\alpha$ and $   P'_{[\beta+1,\alpha)}\subseteq pP'_{[\beta,\alpha)} +R_\alpha.$ Both of them follow from the fact that $pe_{\beta}\in R_\alpha$ and $pe_{\beta,\alpha_2,\dots,\alpha_n} - e_{\alpha_2,\dots,\alpha_n} \in R_\alpha$ for any $\beta+1\leq \alpha_2<\dots<\alpha_n<\alpha, n\geq 2.$ 

It is easy to see that  $\bigcap_{\beta<\lambda} P'_{[\beta,\alpha)} = P'_{[\lambda,\alpha)}$ for any limit ordinal $\lambda.$ Further we claim that for any limit ordinal $\lambda$ we have
\begin{equation}\label{eq:P'_geq2}
\bigcap_{\beta<\lambda} (P'_{[\beta,\alpha)} + R_\alpha) = P'_{[\lambda,\alpha)} + R_\alpha.    
\end{equation}
The inclusion $\supseteq$ is obvious. Lets check the inclusion $\subseteq.$ Take an element $t$ from the intersection. By Lemma \ref{lemma:presentation} there is a unique presentation $t=t_{0}+r$ such that $r\in R_\alpha, t_0\in P'_\alpha,$ ${\rm li}(t_0),{\rm li}(r)\leq {\rm li}(t)$ and coordinates of $t_0$ are from $\{0,\dots,p-1\}.$ Moreover, since $t\in P'_{[\beta,\alpha)},$ then $t_0\in P'_{[\beta,\alpha)}$ for any $\beta<\lambda.$ Therefore, $t_0\in P'_{[\lambda,\alpha)}.$ The assertion follows. 

Let us prove that $p^\beta D'_\alpha = {\rm Im}(P'_{[\beta,\alpha)} \to D'_\alpha).$ The lattice of submodules of $D'_\alpha$ is isomorphic to the lattice of submodules of $P'_\alpha$ containing $R_\alpha.$ The isomorphism is given by taking the preimage. Comparing the definition of $p^\beta D'_\alpha$ and the formulas \eqref{eq:P'_geq1}, \eqref{eq:P'_geq2}, we see that $p^\beta D'_\alpha$ corresponds to $P'_{[\beta,\alpha)}+R_\alpha.$ 

Now let us prove by induction that $p^\beta D'_\alpha = (p^\beta D_\alpha) \cap D'_\alpha.$ For $\beta =0$ it is obvious. 
Assume that $p^\beta D'_\alpha = (p^\beta D_\alpha) \cap D'_\alpha$ and prove that $p^{\beta+1} D'_\alpha = (p^{\beta+1} D_\alpha) \cap D'_\alpha.$ 
The inclusion $\subseteq$ is obvious. 
Let us prove $\supseteq.$ Take $b\in (p^{\beta+1} D_\alpha) \cap D'_\alpha.$ Then $b=p\tilde b$ for $\tilde b\in p^\beta D_\alpha.$ 
Take a preimage $\tilde t\in P_\alpha$ of $\tilde b.$ 
Hence $p \tilde t $ is a preimage of $b\in D'_\alpha.$ 
Thus $p\tilde t\in P'_\alpha.$ 
Therefore $\tilde t\in P'_\alpha.$ 
It follows that $\tilde b\in D'_\alpha.$ 
Hence $\tilde b\in D'_\alpha \cap p^\beta D_\alpha = p^\beta D'_\alpha$ and $b=p \tilde b\in p^{\beta+1}D'_\alpha.$ So we proved $p^{\beta+1} D'_\alpha = (p^{\beta+1} D_\alpha) \cap D'_\alpha.$ Now assume that for a limit ordinal $\lambda$ for any $\beta<\lambda$ we have  $p^\beta D'_\alpha = (p^\beta D_\alpha) \cap D'_\alpha.$ Since intersection commutes with intersection, we obtain $p^\lambda D'_\alpha = (p^\lambda D_\alpha) \cap D'_\alpha.$ 
\end{proof}

\begin{remark}
Lemma \ref{lemma:p^betaD'} implies that $(p^\alpha D_\alpha)\cap D'_\alpha=0.$ However, in general $p^\alpha D_\alpha \ne 0.$ Moreover, we claim that 
\begin{equation}
p^\infty D_\omega \ne 0.    
\end{equation}
Let us give a sketch of a proof here. Consider the group $A=\ZZ^{\prod \omega}/\ZZ^{\bigoplus \omega}.$ It is easy to see that the element $(p,p^2,p^3,\dots)$ of $A$ lies in $p^\infty A.$ Hence $p^\infty A\ne 0.$ Consider a map $\varphi: A\to D_\omega$ defined by
$\varphi(n_0,n_1,\dots) = \sum_{i\in \omega} n_i p^{i+1} e_{0,1,\dots,i}.$ Here each summand of the ``infinite sum'' is equal to zero $p^{i+1}e_{0,1,\dots,i} =0$ but the whole ``infinite sum'' is not zero. It is easy to check that $\varphi$ is a well defined monomorphism. Therefore $p^\infty D_\omega \neq 0.$
\end{remark}

\begin{proof}[Proof of Theorem \ref{th:example}] First we prove that $S(E_\alpha)$ is local. By the definition of $E_\alpha$ we have $p^\alpha E_\alpha=0.$ Then \eqref{eq:limS(A)} implies that $\lim S(E_\alpha)=0.$ The inverse sequence $S(\ZZ_p)$ is local, because ${\rm len}(S(\ZZ_p))=\omega,$ $I^\omega(S(\ZZ_p))=0,$ $S(\ZZ_p)\to \widehat S^\omega (\ZZ_p)$ is an isomorphism (Theorem \ref{th:gen_ML}, Proposition \ref{prop:mu}). Since the class of local inverse sequences is closed with respect to small limits, we obtain that $S(P_\alpha)$ is also local. Using that there is an epimorphism $P_\alpha\epi E_\alpha,$ we obtain $\lim^1 S(E_\alpha)=0.$ Now we prove that ${\rm len}(S(E_\alpha))=\alpha.$ Equivalently, we need to prove that ${\rm len}_p(E_\alpha)=\alpha.$ Since $p^\alpha E_\alpha=0,$ we get ${\rm len}_p(E_\alpha)\leq \alpha.$ Now we need to prove that $p^\beta E_\alpha\ne 0$ for any $\beta<\alpha.$ It is sufficient to prove that $p^\beta D_\alpha\ne p^\alpha D_\alpha$ for any $\beta<\alpha.$ By Lemma \ref{lemma:p^betaD'} we have $e_\beta +R_\alpha \in  (p^\beta D_\alpha)\cap D'_\alpha$ and $(p^\alpha D_\alpha)\cap D'_\alpha=0.$  Lemma \ref{lemma:presentation} implies that $e_\beta \notin R_\alpha.$ Therefore $(p^\beta D_\alpha)\cap D'_\alpha\neq (p^\alpha D_\alpha)\cap D'_\alpha.$
\end{proof}

\subsection{A local inverse sequence, which is not complete with respect to a regular cardinal}

In this subsection for any regular uncountable cardinal $\kappa$ we will construct an abelian group $A$ such that $S(A)$ is local and not $\kappa$-complete. 

Let $\kappa$ be a regular cardinal. For a family of abelian groups $(A_x)_{x\in X}$ we define its $\kappa$-supported product as the subgroup $\prod_{x\in X}^{(\kappa)}A_x$ of the product $\prod_{x\in X}A_x$ consisting of elements whose support has cardinality $<\kappa.$ In more categorical terms we can define the $\kappa$-supported product as the $\kappa$-filtered colimit of products taken by all subsets of $X$ of cardinality $<\kappa$
\begin{equation}
{\prod}_{x\in X}^{(\kappa)}A_x = \underset{K\subset_{\kappa}X}{\rm colim} \left( {\prod}_{x\in K} A_x \right).
\end{equation}
For example, the direct sum is the $\aleph_0$-supported direct product. It is easy to see check that for any ordinal $\alpha$ we have

\begin{lemma}\label{lemma:kappa-prod} The $\kappa$-supporting product satisfies the following elementary properties.
\begin{enumerate}
\item For any ordinal $\alpha$ the following holds
\[
p^\alpha \left({\prod}_{x\in X}^{(\kappa)}A_x\right) = {\prod}_{x\in X}^{(\kappa)} (p^\alpha A_x).
\]

\item If $(A_{x,y})_{(x,y)\in X\times Y}$ is a family indexed by product $X\times Y$, then there is an isomorphism
\[ {\prod}_{x\in X}^{(\kappa)} \left ({\prod}_{y\in Y}^{(\kappa)} A_{x,y}\right) \cong {\prod}_{(x,y)\in X\times Y}^{(\kappa)} A_{x,y}  \cong  {\prod}_{y\in Y}^{(\kappa)} \left ({\prod}_{x\in X}^{(\kappa)} A_{x,y}\right).\]

\item  If $|Y|<\kappa,$ then 
\[ {\prod}_{x\in X}^{(\kappa)} \left ({\prod}_{y\in Y} A_{x,y}\right)  \cong  {\prod}_{y\in Y} \left ({\prod}_{x\in X}^{(\kappa)} A_{x,y}\right).\]

\item If $J$ is a small category such that $|{\rm Ob}(J)|,|{\rm Mor}(J)|<\kappa,$ then for any family of functors to the category of abelian groups $(A_x: J \to {\rm Ab})_{x\in X}$ there is a natural isomorphism 
\[ {\prod}_{x\in X}^{(\kappa)} \left (\lim_J A_x \right)  \cong  \lim_J\left ({\prod}_{x\in X}^{(\kappa)} A_{x,y}\right).\]

\item If $\kappa$ is uncountable cardinal and there is a family of inverse sequences $(S_x)_{x\in X}$, then there is a natural isomorphism 
\[
{\lim}^1 \left({\prod}_{x\in X}^{(\kappa)} S_x \right) \cong {\prod}_{x\in X}^{(\kappa)} \left( {\lim}^1 S_x \right).  
\]
\end{enumerate}
\end{lemma}
\begin{proof}
(1). Straightforward proof by transfinite induction. 

(2). Follows from the fact that $\kappa$ is regular, and hence,  a union of sets of cardinality $<\kappa$ indexed by a set of cardinality $<\kappa$ has cardinality $<\kappa.$

(3). Follows from (2).

(4). It follows from (3) and the fact that $\lim_J$ can be presented as the equalizer of two natural maps $\prod_{j\in {\rm Ob}(J)} A_x(j) \to \prod_{\alpha\in {\rm Mor}(J)} A_x({\rm cod}(\alpha)).$

(5). It follows from (3) and the exact sequence \eqref{eq:lim_es}.  
\end{proof}

$\kappa$-supported products of inverse sequences are defined component-wise. 

\begin{corollary}\label{cor:supported_product}
Let $\kappa$ be a regular uncountable cardinal and $(S_{x})_{x\in X}$ be a family of local inverse sequences of abelian groups. Then the $\kappa$-supported product ${\prod}^{(\kappa)}_{x\in X} S_{x}$ is local.
\end{corollary}

\begin{theorem}\label{th:example2} Let $\kappa$ be a regular uncountable cardinal,  $A=\prod^{(\kappa)}_{\alpha<\kappa} E_{\alpha+1}$ and $S=S(A)$. Then $S$ is local, not $\kappa$-complete and ${\rm len}(S)=\kappa$. 
\end{theorem}
\begin{proof} In the proof we use  simplified notations   $E_\alpha^\beta = E_\alpha/p^\beta E_\alpha$ and $A^\beta = A/p^\beta A.$ Corollary \ref{cor:supported_product} and Theorem \ref{th:example} imply that $S(A)$ is local. Lemma \ref{lemma:kappa-prod}(1) implies that for any ordinal $\beta$ we have $p^\beta A= {\prod}^{(\kappa)}_{\alpha<\kappa} p^\beta E_{\alpha+1}$ and $A^\beta = {\prod}^{(\kappa)}_{\alpha<\kappa} E^\beta_{\alpha+1}.$ Since ${\rm len}_p(E_{\alpha+1})=\alpha+1,$ we obtain that ${\rm len}(A)=\kappa.$  We need to prove that $A\to \widehat A^\kappa$ is not surjective. Any element of $A$ is a family $(e_\alpha)_{\alpha<\kappa},$ where $e_\alpha\in E_{\alpha+1}$ and $|\{\alpha\mid e_\alpha\ne 0\}|<\kappa.$ Using the isomorphism 
\begin{equation}
\widehat A^\kappa \cong \lim_{\beta<\kappa} \left( {\prod}^{(\kappa)}_{\alpha<\kappa} E^\beta_{\alpha+1} \right) 
\end{equation}
we can present elements of $\widehat A^\kappa$ as families $(f_{\alpha,\beta})_{\alpha,\beta<\kappa}
$ such that
\begin{itemize}
    \item[(i)] $f_{\alpha,\beta}\in E^\beta_{\alpha+1};$
    \item[(ii)]  for any  $\beta' < \beta < \kappa$ and any $\alpha$ the map $E^\beta_{\alpha+1} \to E^{\beta'}_{\alpha+1}$ sends $f_{\alpha,\beta}$ to $f_{\alpha,\beta'};$
    \item[(iii)] for any $\beta<\kappa$ we have $|\{\alpha\mid 
f_{\alpha,\beta}\ne 0 \}|<\kappa.$
\end{itemize}
Then the map $A\to \widehat A^\kappa$ sends $(e_\alpha)_{\alpha<\kappa}$ to 
$(e_{\alpha}+p^\beta E_{\alpha+1})_{\alpha,\beta}.$
Using this description it is easy to see any element $(f_{\alpha,\beta})$ from the image $A\to \widehat A^\kappa$  satisfies the following property 
\begin{itemize}
    \item[(iv)]  $|\{\alpha \mid \exists \beta : f_{\alpha,\beta}\ne 0  \}|<\kappa.$
\end{itemize}
Now we construct an element from $\widehat A^\kappa$ which is not from the image of $A\to \widehat A^\kappa.$ Take a non-zero element $x_\alpha\in p^\alpha E_{\alpha+1}$ and consider a family $(g_{\alpha,\beta})$ such that $g_{\alpha,\beta}=x_{\alpha}+p^\beta E_{\alpha+1}.$ The family $(g_{\alpha,\beta})$ lies in $\widehat A^\kappa$ because it satisfies (i),(ii),(iii).  But it does not lie in the image of $A\to \widehat A^\kappa$ because $\{\alpha \mid \exists \beta : g_{\alpha,\beta}\ne 0  \}=\kappa$ and it does not satisfy (iv).
\end{proof}

\begin{corollary}
If $A=\prod^{(\aleph_1)}_{\alpha<\aleph_1} E_{\alpha+1},$ then $S(A)$ is local and $\lambda$-complete for any limit ordinal $\lambda\neq \aleph_1,$ but not $\aleph_1$-complete. 
\end{corollary}

\printbibliography
\end{document}